\documentclass[reqno,11pt]{amsart}
\usepackage{amsmath,amssymb,mathrsfs,amsthm,amsfonts}
\usepackage[inline]{enumitem}
\usepackage[usenames,dvipsnames]{xcolor}
\usepackage{hyperref}
\usepackage[percent]{overpic}
\usepackage{comment}
\usepackage{algorithm}
\usepackage{algorithmic}
\usepackage{mathtools}
\usepackage{subcaption}
\usepackage{tikz}
\usepackage{bm}
\usetikzlibrary{calc}
\usetikzlibrary{arrows.meta,backgrounds}
\usepackage{multirow,array,longtable,booktabs}
\usetikzlibrary{arrows}

\hypersetup{
	colorlinks=true, linkcolor=blue,
	citecolor=ForestGreen
}
\usepackage[paper=letterpaper,margin=1in]{geometry}
\DeclareMathAlphabet{\mathpzc}{OT1}{pzc}{m}{it}

\newtheorem{theorem}{Theorem}[section]
\newtheorem{lemma}[theorem]{Lemma}
\newtheorem{proposition}[theorem]{Proposition}

\theoremstyle{definition}
\newtheorem{definition}[theorem]{Definition}

\newtheorem{assumption}[theorem]{Assumption}
\newtheorem{remark}[theorem]{Remark}
\numberwithin{equation}{section}
\allowdisplaybreaks
\usepackage{acronym}

\acrodef{LDP}{Large Deviation Principle}

\newcommand{\be}{\begin{equation}}
\newcommand{\ee}{\end{equation}}
\newcommand{\bea} {\begin{array}{rl}}
\newcommand{\eea} {\end{array}}
\newcommand{\bepa}{\left\{ \begin{array}{l}}
\newcommand{\eepa} {\end{array}\right.}
\makeatletter

\newcommand\norm[1]{\left\Vert#1\right\Vert}

\newcommand{\R}{\mathbb{R}} 
\newcommand{\N}{\mathbb{N}}

\newcommand{\e}{\epsilon}

\newcommand{\T}{\mathbb{T}}

\usepackage{graphicx}

\newcommand{\bd}{{\bf d}}
\newcommand{\ov}{\overline}
\newcommand{\bx}{\bm{x}}
\newcommand{\bX}{\bm{X}}
\newcommand{\bY}{\bm{Y}}
\newcommand{\bP}{\mathbb{P}}
\newcommand{\cP}{\mathcal{P}}
\newcommand{\sL}{\mathcal{L}}
\newcommand{\eps}{\epsilon}
\newcommand{\cent}{C_{\text{ent}}}
\newcommand{\cconc}{C_{\text{con,p}}}
\newcommand{\abs}{\text{ac}}
\newcommand{\linf}{L^{\infty}}

\title{Concentration bounds for stochastic systems with singular kernels}

\author[J.\ Jackson]{Joe Jackson}
\address{J.\ Jackson,
	Department of Mathematics, University of Chicago,
	\newline\hphantom{\quad \ \ J. Jackson}
	5734 S.~University Avenue, Chicago, Illinois 60637 USA
}
\email{jsjackson@uchicago.edu}

\author[A.\ Zitridis]{Antonios Zitridis}
\address{A.\ Zitridis,
	Department of Mathematics, University of Chicago,
	\newline\hphantom{\quad \ \ A. Zitridis}
	5734 S.~University Avenue, Chicago, Illinois 60637 USA
}
\email{zitridisa@uchicago.edu}

\begin{document}
	\begin{abstract}
            This note is concerned with weakly interacting stochastic particle systems with possibly singular pairwise interactions. In this setting, we observe a connection between entropic propagation of chaos and exponential concentration bounds for the empirical measure of the system. In particular, we establish a variational upper bound for the probability of a certain rare event, and then use this upper bound to show that ``controlled" entropic propagation of chaos implies an exponential concentration bound for the empirical measure. This connection allows us to infer concentration bounds for a class of singular stochastic systems through a simple adaptation of the arguments developed in \cite{jabin2018quantitative}.
	\end{abstract}
	
	
	\maketitle
	{
		}

\section{Introduction}

\subsection{Stochastic particle systems}

Let $N,d\in \mathbb{N}$ and $K:\mathbb{T}^d\rightarrow \R^d$ be a vector field defined on the $d-$dimensional flat torus $\mathbb{T}^d$. We consider the system of $N$ particles described by the dynamics
\be \label{system}
dX_t^i = \Big( F(X_t^i) + \frac{1}{N}\sum_{j\neq i}K(X_t^i-X_t^j) \Big)dt+\sqrt{2}dW_t^i,\;\; t \in [0,T], \,\, i=1,2,...,N,
\ee
where the $W^i$ are $N$ independent standard Wiener processes defined on a standard filtered probability space $\big(\Omega, \mathbb{F} = (\mathcal{F}_t)_{0 \leq t \leq T}, \bP)$. We work on a finite time horizon $[0,T]$, and we assume that the initial positions of the particles are i.i.d., i.e. 
\begin{align*}
    \bX_0 = (X_0^1,...,X_0^N) \sim \rho_0^{\otimes N} \text{ for some $\rho_0 \in \cP(\T^d)$.}
\end{align*}
Thus the data used to define our particle system consists of the time horizon $T > 0$, the maps $F, K : \T^d \to \R^d$, and the initial distribution $\rho_0 \in \cP(\T^d)$. Under these conditions, the law $\rho^N_t = \sL(\bX_t)$\footnote{Here and throughout the note we use the notation $\rho^N_t$ to indicate a curve $[0,T] \;\reflectbox{$\in$}\; t \mapsto \rho^N_t \in \cP(\T^{dN})$, and we abuse notation by writing $\rho^N(t,\cdot)$ for the density of $\rho^N_t$ when it exists.} of the particles is formally described by the Liouiville equation
\vspace{-4mm}

\begin{align} \label{Liouville fo}
\partial_t\rho^N+\sum_{i=1}^N\text{div}_{x^i}\bigg( \rho^N \Big(F(x^i) + \frac{1}{N}\sum_{j \neq i}K(x^i-x^j) \Big)\bigg)-\sum_{i=1}^N\Delta_{x_i}\rho^N=0, \qquad \rho^N_0 = \rho_0^{\otimes N},
\end{align}
where $\bx=(x^1,...,x^N) \in \T^{dN}$. When $F$ and $K$ are smooth (or at least Lipschitz), it is well known that the asymptotic behaviour of the particle system \eqref{system} is described by the non-local Fokker-Planck equation 
\be \label{limit pde}
\partial_t \overline{\rho}+\text{div}\big(\overline{\rho} \, (F + K*\overline{\rho}) \big)-\Delta \overline{\rho}=0, \qquad 
\ov{\rho}_0 = \rho_0.
\ee
More precisely, it is known that in an appropriate sense we have
\begin{align*}
    m_{\bX_t}^N = \frac{1}{N} \sum_{i = 1}^N \delta_{X_t^i} \xrightarrow{N\to \infty} \ov{\rho}_t, \quad \sL(X_t^1,...,X_t^k) \xrightarrow{N \to \infty} \ov{\rho}_t^{\otimes k}, \text{ for each fixed k}.
\end{align*}
The first statement above confirms that $\ov{\rho}_t$ is the mean field limit of the empirical measures $m_{\bX_t}^N$, while the second statement is referred to as propagation of chaos.
Quantitative versions of these statements, as well as finer results like central limit theorems, large deviations principles, and concentration bounds for the empirical measures $m_{\bX_t}^N$ have also been obtained for regular $K$. We do not make any attempt to summarize this literature, but refer to survey articles like \cite{jabinwangsurvey, chaintron2022} for an introduction. Extending such results to more singular kernels $K$, which are often met in applications, is a very active area of research. We mention in particular the recent flurry of activity around singular kernels which derive from Riesz potentials (see e.g. \cite{serfatycoulomb, breschjabinwang, rosserfaty2023, decourcel2023} and the references therein) and some recent efforts aimed at singular attractive kernels (see e.g. \cite{breschjabinwangatt, decourcelattractive}). More relevant to the present note are the slightly less recent contributions of \cite{jabin2016mean} and \cite{jabin2018quantitative}, where quantitative propagation of chaos was established by estimating the relative entropy between a solution of the Liouiville equation \eqref{Liouville fo} and the tensor product $\overline{\rho}^N=\overline{\rho}^{\otimes N}$ of the solution of \eqref{limit pde}. More precisely, \cite{jabin2018quantitative} established the following quantitative version of ``entropic propagation of chaos":
\be \label{entest}
H_N(\rho^N_T|\overline{\rho}^N_T)\leq H_N(\rho_0^N|\overline{\rho}_0^N)+ C/N,
\ee
for some constant $C$ independent of $N$, where $H$ denotes the rescaled relative entropy
$$H_N(\rho^N_t|\overline{\rho}^N_t):=\frac{1}{N}H(\rho^N_t|\overline{\rho}^N_t)=\frac{1}{N}\int_{\mathbb{T}^{dN}}\rho^N(t,\bx)\log\frac{\rho^N(t,\bx)}{\overline{\rho}^N(t,\bx)}d\bx.$$
We note that \eqref{entest} is established under a general initial condition for \eqref{Liouville fo}, but if $\rho^N_0 = \ov{\rho}_0^{\otimes N}$ then the first term on the right-hand side of \eqref{entest} vanishes, leading to $H_N(\rho^N_T| \ov{\rho}^N_T) = O(1/N)$, which by a classical subadditivity property of relative entropy gives $H(\rho^{N,k}_T | \ov{\rho}^{\otimes k}_T) = O(k/N)$, where $\rho^{N,k}$ denotes the $k$-particle marginal of $\rho^N$, and in particular $H(\rho^{N,k}_T | \ov{\rho}^{\otimes k}_T) \to 0$ as $N \to \infty$ for each fixed $k$, which is what we mean by entropic propagation of chaos. We note also that establishing \eqref{entest} is not the only way to obtain quantitative entropic propagation of chaos - we refer to \cite{lackerhierarchies} for a more detailed discussion and for an approach which leads to optimal bounds on $H(\rho_T^{N,k} | \ov{\rho}_T^{\otimes k})$.

\subsection{Our results}

Our goal in this note is to point out a connection between entropic propagation of chaos and concentration inequalities for the empirical measure of the particle system \eqref{system}. In particular, for $1 \leq p < \infty$, we are interested in bounds of the form 
\begin{align} \label{concentration}
    \bP\Big[ \bd_p(m^N_{\bX_T}, \ov{\rho}_T) > \eps \Big] \leq \cconc \exp( - \cconc^{-1} a_p(\eps) N), \quad a_p(\eps) \coloneqq 
    \begin{cases}
        \eps^{2p} & p > d/2, \\
        \eps^{2p}/\big(\log(2 + 1/\eps^p)\big)^2 & p = d/2,\\
        \eps^{d} & p < d/2, 
    \end{cases}
\end{align}
where $\bd_p$ denotes the $p$-Wasserstein distance.
That is, we want to show that the empirical measure for the particle system admits concentration bounds of the same type available for i.i.d. samples, as established in \cite{FournierGuillin}. When written in terms of the Liouville equations (which is often necessary for technical reasons when $K$ is singular), \eqref{concentration} becomes
\begin{align} \label{concentrationliouville}
    \rho^N_T( A^p_{N,\e}) \leq \cconc \exp(- \cconc^{-1} a_p(\eps) N ), \text{ where }A^p_{N,\e}=\bigg\{ \bx\in \mathbb{T}^{dN}\bigg| {\bf d}_p(m^N_{\bx},\overline{\rho}^N_T)>\e \bigg\}.
\end{align}
Such concentration inequalities were obtained for regular interactions by several authors. We highlight in particular \cite{Bolley2005QuantitativeCI}, where concentration inequalities for the system \eqref{system} were inferred from concentration inequalities for i.i.d. random variables via the so-called ``synchronous coupling" method originally due to McKean \cite{mckean} and popularized by Sznitman \cite{sznitman}, and \cite{Delarue2018FromTM}, where a more general concentration of measure result was obtained by exploiting certain (uniform in $N$) functional inequalities for the law of $\bX = (X^1,...,X^N)$. When $K$ is singular, the approaches in \cite{Bolley2005QuantitativeCI} and \cite{Delarue2018FromTM} break down, and the only estimate similar to \eqref{concentration} we are aware of is Proposition 5.3 of \cite{hesschilds}, where the structure of the repulsive Riesz interactions is leveraged to get bounds somewhat similar to \eqref{concentration}.

To explain our main idea, we must introduce the ``controlled Liouville equation'':
\be \label{Liouville foc}
\partial_t f^N+\sum_{i=1}^N\text{div}_{x^i}\bigg( f^N\Big(\alpha^i(t,\bx)+ F(x^i) + \frac{1}{N}\sum_{j\neq i} K(x^i-x^j)\Big)\bigg)-\sum_{i=1}^N\Delta_{x^i}f^N=0,
\ee
where $\bm \alpha =(\alpha^1,...,\alpha^N) : [0,T] \times (\T^d)^N \to (\R^d)^N$ is a measurable map which we view as a control. For the so-called $W^{-1,\infty}$ kernels of \cite{jabin2018quantitative}, one can (as is explained in more detail in the proof of Proposition \ref{prop.wminus}) easily generalize the original entropy estimate \eqref{entest} of Jabin and Wang to get an estimate of the form\footnote{The factor of $1/4$ in \eqref{controlledest} is purely aesthetic, and is included to make \eqref{controlledest} more consistent with Proposition \ref{varformula} below.}
\vspace{-2mm}
\be \label{controlledest}
H_N(f^N_T|\overline{\rho}^N_T)\leq H_N(f^N_0|\overline{\rho}^N_0)+\cent \bigg( \frac{1}{N}+\frac{1}{4N}\sum_{i=1}^N\int_0^T\int_{\mathbb{T}^d}|\alpha^i(t,\bx)|^2 df^N_t(\bx) dt\bigg),
\ee
for any solution $f^N$ to \eqref{Liouville foc}. We emphasize that here $f^N_0$ may be different from $\rho_0^{\otimes N}$. We think of \eqref{controlledest} as a ``controlled" or ``perturbed" version of the entropic propagation of chaos estimate \eqref{entest} of Jabin and Wang. Our main observation is that an estimate of the form \eqref{controlledest} in fact implies a concentration bound of the form \eqref{concentration}. For technical reasons, we first make this observation precise through the following a-priori estimate, which requires $K$ and $\ov{\rho}^0$ to be bounded.

\begin{theorem} \label{thm.apriori}
    There are constants $C_{d,p}$ depending only on $d$ and $p$ such that the following holds: Suppose that $K$ and $F$ are bounded, the initial condition $\rho_0$ has a bounded density, and that there exists a solution $\ov{\rho}$ to \eqref{limit pde} and a constant $\cent \geq 1$ such that \eqref{controlledest} holds for all $\bm \alpha$ and $f^N$ such that $f^N$ is an entropy solution of \eqref{Liouville foc} as in Definition \ref{def entropy} below. Then the concentration bound \eqref{concentration} holds for each $p$, with $\cconc = C_{d,p} \cent$.
\end{theorem}

\begin{remark}
    We note that because $K$ and $F$ are bounded and there is non-degenerate idiosyncratic noise, there is no problem making sense of the SDE \eqref{system} or the Liouville equation \eqref{Liouville fo}.
\end{remark}

\begin{remark}
    While $K$ is required to be bounded for technical reasons, the key point of the a-priori estimate is that $\cconc$ depends on $K$ \textit{only} through the constant $\cent$ appearing in \eqref{controlledest}, and not on $\|K\|_{\infty}$. Similarly, there is no explicit dependence of our a-priori estimate on $T$, so if one establishes \eqref{entest} with a constant $\cent$ independent of $T$ (which can be done for the 2-D stochastic vortex model with the techniques of \cite{guillin2024}) then one gets a uniform in time concentration bound.
\end{remark}

\begin{remark}
    Notice that in Theorem \ref{thm.apriori}, we only require \eqref{entest} to hold for \textit{entropy} solutions of \eqref{Liouville foc}, rather than for all weak solutions. This creates some technical challenges in the proof, but it is essential for the application to $W^{-1,\infty}$ kernels below. 
\end{remark}

As mentioned above, we can verify that \eqref{controlledest} holds under roughly the same conditions as in \cite{jabin2018quantitative}. Thus we make use of the following assumption (and refer to the notation section below for the definition of $W^{-1,\infty}$):
\begin{assumption} \label{assump.main}
    For some $\beta \in (0,1)$, $C_0 > 0$, we have $F \in C^{1,\beta}$, $\rho_0 \in C^{2,\beta}$, $K$, $\text{div} \, K \in W^{-1,\infty}$, and 
    \begin{align*}
        \big( \inf_x \rho_0(x) \big)^{-1} + \|\rho_0\|_{C^{2,\beta}} + \|F\|_{C^{1,\beta}} + \|K\|_{-1,\infty} + \|K\|_{L^1} + \|\text{div} \, K\|_{-1,\infty} \leq C_0. 
    \end{align*}
\end{assumption}
\begin{proposition}
    \label{prop.wminus}
    Let Assumption \ref{assump.main} hold. Then there is a unique classical solution $\ov{\rho}$ to \eqref{limit pde}, and for each $N \in \N$ there exists an admissible entropy solution $\rho^N$ to \eqref{Liouville fo} in the sense of Definition \ref{def.admissiblesoln}. Moreover, for each $p \in [1,\infty)$ there is a constant $\cconc$ which depends only on $d$, $p$, $T$ and $C_0$, such that any admissible entropy solution $\rho^N$ of \eqref{Liouville fo} satisfies \eqref{concentrationliouville}. 
\end{proposition}

\begin{remark}
    Regarding possible extensions: one can easily extend Theorem \ref{thm.apriori} by replacing $\T^d$ by $\R^d$, but the concentration bound will then depend on the (exponential) moments of the limit $\ov{\rho}$. We focus on the periodic case because in this case we can adapt the arguments of \cite{jabin2018quantitative} to obtain \eqref{entest}. However, the recent preprint \cite{wangfullspace} shows that the program of \cite{wangfullspace} can be carried out in the whole space for the 2-D viscous vortex model, so the (non-periodic analogue of) our Theorem \ref{thm.apriori} can be combined with the argument in \cite{wangfullspace} to obtain an anologue of Proposition \ref{prop.wminus} in this setting. Likewise, one can easily adapt Theorem \ref{thm.apriori} to the kinetic setting, in which case the ``controlled Liouiville equation" will involve a control only in the velocity (as is easily seen from an inspection of the proof of Lemma \ref{varformula}, in particular the application of Girsanov's Theorem). By combining such a result with the techniques of \cite{jabin2016mean}, one can easily derive an anologue of Proposition \ref{prop.wminus} for kinetic particle systems with bounded forces, i.e. in the same setting as \cite{jabin2016mean}.
\end{remark}

\section{Notation and Preliminaries}

\subsection{Notation and some definitions}
Throughout the paper $T$ is a positive real number and $\mathbb{T}^d$ is the $d-$dimensional flat torus. $f*g$ is the convolution between the functions $f,g$. If a function $u:[0,T]\times\mathbb{T}^d\rightarrow \R$ is sufficiently regular, then we denote by $\partial_t u,\; D_{x_i }u$ the partial derivative with respect to $t$ and the partial derivative with respect to $x_i$ (the $i$-th coordinate) for $i=1,2,...,d$, respectively. For $p\in [1,+\infty]$, we use the symbol $L^p_{t,x}$ (resp. $L^p_x)$ for the space of functions $f$ such that $\int |f(t,x)|^pdtdx<+\infty$ (resp. $\int |f(x)|^pdx<+\infty$). For $k\geq 1$, $p\in [1,+\infty]$ and $\beta\in (0,1)$, we denote by $W^{k,p}$, $C^{k,\beta}$ the Sobolev ($k$ weak derivatives in $L^p$) and H\"older spaces ($k$ derivatives which are $\beta$-H\"older continuous) on $\T^d$, respectively. Also, we write $C^{k,\beta}_{t,x}$ for the standard parabolic H\"older spaces on $[0,T] \times \T^d$, e.g. $C^{2,\beta}_{t,x}$ will be the space of functions $f(t,x)$ such that $\partial_t f$, $Df$, $D^2f$ exist, and are $\beta/2$-H\"older in $t$ and $\beta$-H\"older in $x$. Similarly, we will use $W^{k,p}_{t,x}$ for  the standard parabolic H\"older spaces, e.g. $W^{2,p}_{t,x}$ will indicate the space of functions $f(t,x)$ such that $\partial_t f$, $Df $, $D^2f$ are in $L^p$.

\noindent
$\mathcal{P}(\mathbb{T}^d)$ (resp. $\mathcal{P}(X)$) is the space of probability measures over $\mathbb{T}^d$ (resp. a Polish space $X$), and $\cP_{\abs}(\T^{dN}) \subset \cP(\T^{dN})$ is the set of probability measures which admit a density with respect to the Lebesgue measure. For $p\geq 1$, the $p-$Wasserstein space is denoted by $\mathcal{P}_p(\mathbb{T}^d)$ and its metric is ${\bf d }_p$. Given a curve $[0,T]\; \reflectbox{$\in$}\; t \mapsto m_t \in \cP(\T^d)$, we write $L^2_{dt\otimes m_t}([0,T]\times \mathbb{T}^d,\R^d)$ for the space of $\R^d$-valued $m_t \otimes dt$-square integrable functions over $[0,T]\times\mathbb{T}^d$. The space of $\R^d$-values Borel measures over $[0,T]\times \mathbb{T}^d$ with finite total variation is denoted by $\mathcal{M}([0,T]\times\mathbb{T}^d,\R^d)$. For any $N\in\mathbb{N}$, $\mathbb{U}_N$ is the uniform distribution on $\mathbb{T}^{dN}$; we will write $\mathbb{U}$ when there is no confusion.
The relative entropy $H(\mu |\nu)$ of two probability measures $\mu,\nu$ is defined as follows
$$H(\mu|\nu)=\begin{cases}
    \int_{\mathbb{T}^d} \frac{d\mu}{d\nu}\log(\frac{d\mu}{d\nu})d\nu(x),&\text{ if } \mu \ll\nu,\\
    +\infty,& \text{ otherwise.}
\end{cases}$$
For simplicity, we write $H(\mu) = H(\mu | \mathbb{U})$. For solutions of the Liouville equation \eqref{Liouville fo} we will be using the symbol $\rho^N_t$ or $\rho^N(t)$ to denote an element of $\mathcal{P}_{\abs}(\T^d)$ and the symbol $\rho^N(t,x)$ for the corresponding density, which we can view as an element of $L^{\infty}_tL^1_x$. We use the analogous notation for $f^N$, a solution of the ``perturbed'' Liouville equation \eqref{Liouville foc}. We will be using the notation $I(\rho)=\int_{\T^{dN}}\frac{|D_{\bx}\rho|^2(\bx)}{\rho(\bx)}d\bx$ for the Fisher information and, for $j=1,...,N$, $I_j(\rho)=\int_{\T^{dN}}\frac{|D_{x^j}\rho|^2(\bx)}{\rho(\bx)}d\bx$. 

We now recall the precise definition of the space $W^{-1,\infty}$.
\begin{definition}
(i) A function $f$ with $\int_{\mathbb{T}^d}f=0$ belongs to $W^{-1,\infty}(\mathbb{T}^d)$ if and only if there exists a vector field $V\in L^{\infty}(\mathbb{T}^d)$ such that $f=\text{div} \, V$. We denote 
$$\|f\|_{-1,\infty}=\inf_V\|V\|_{L^{\infty}},\; \text{ where }f=\text{div}\,V.$$
(ii) A vector field $K$ with $\int_{\mathbb{T}^d}K=0$ belongs to $W^{-1,\infty}(\mathbb{T}^d)$ if and only if there exists a matrix field $V \in L^{\infty}(\mathbb{T}^d)$ such that $K=\text{div} \, V$ in the sense that $K_i=\sum_j \partial_j V_{ij}$. We denote 
$$\|K\|_{-1,\infty}=\inf_V\|V\|_{L^{\infty}},\; \text{ where }K=\text{div} \,V.$$
\end{definition}

 As in \cite{jabin2018quantitative}, since $K$ (and possibly $\bm \alpha$) are not smooth, in order to get controlled entropy bounds between $\rho^N$ and $\overline{\rho}^N=\overline{\rho}^{\otimes N}$ we must work with entropy solutions.

\begin{definition}\label{def entropy}
Suppose that $K\in W^{-1,\infty}$ and $\text{div} \, K \in W^{-1,\infty}$, so that there exists a vector field $V \in \linf(\T^d)$ such that $\text{div} \, K = \text{div} \, V$. A continuous map $[0,T] \;\reflectbox{$\in$}\; t \mapsto f_t^N \in \cP_{\abs}(\T^{dN})$ is an entropy solution to \eqref{Liouville foc} on the time interval $[0,T]$ if $f^N$ solves \eqref{Liouville foc} in the sense of distributions, $D_{\bx}f^N(t,\cdot)$ exists in the weak sense for a.e $t\leq T$
and for each $0 \leq t \leq T$,
\begin{align}
&H(f_t^N)+\sum_{i=1}^N\int_0^t\int_{\mathbb{T}^{dN}}\frac{|D_{x^i}f^N(s,\bx)|^2}{f^N(s,\bx)}d\bx ds \nonumber
\\
&\qquad \leq H(f_0^N)+\frac{1}{N}\sum_{i,j=1}^N\int_0^t\int_{\mathbb{T}^{dN}}(\alpha^i(s,\bx)+ F(x^i) + V(x^i-x^j))\cdot D_{x^i}f^N_sd\bx ds.\label{Entineq}
\end{align} 
To indicate the dependence on $\bm \alpha,\; F$ and $K$ we call such a solution an $(\bm \alpha,F,K)$-entropy solution. When $\bm \alpha = 0$ and $f_0^N = \rho_0^{\otimes N}$, we simply say that $f^N$ is an entropy solution of \eqref{Liouville fo}.
\end{definition}

\begin{remark} \label{rem1}
(i) If $\alpha=0$, then we have the definition of entropy solution introduced in \cite{jabin2018quantitative}.\\
(ii) If $f^N$ satisfies \eqref{Liouville foc} in the classical sense, then it is also an entropy solution.\\
(iii) It follows that any $(\alpha,F,K)-$entropy solution is also an $(\alpha+\tilde{\beta},F,K-\beta)$-entropy solution, for any bounded vector field $\beta$, where $\tilde{\beta}$ is a vector field such that $\tilde{\beta}^i(\bx)=\frac{1}{N}\sum_{j\neq i}\beta(x^i-x^j)$.\\
(iv) By virtue of Proposition 1 from \cite{jabin2018quantitative}\footnote{Actually, regularity in time is not addressed in Proposition 1 of \cite{jabin2018quantitative}, but it is straightforward to check using the assumptions on $K$ that any entropy solution in the sense of Jabin and Wang admits a version in $C([0,T]; \cP_{\abs}(\T^{dN}))$.}, if $F$ and $\text{div} F$ are bounded and $K,\text{div}K\in W^{-1,\infty}$, then there exists a $(0,F,K)-$entropy solution. 
\end{remark}

For technical reasons, we at times need to work specifically with entropy solutions of \eqref{Liouville fo} which arise via a suitable mollification procedure. In particular, we fix throughout the paper a standard mollifier $(\rho_{\delta})_{\delta > 0}$ on $\T^d$, and we define $K_{\delta} = K * \rho_{\delta}$, $F_{\delta} = F * \rho_{\delta}$. Then we make the following definition.

\begin{definition} \label{def.admissiblesoln}
    We say that $\rho^N$ is an admissible entropy solution of \eqref{Liouville fo} if it is an entropy solution, and for some $\delta_k \downarrow 0$, we have $\rho^{N,\delta_k}_t \xrightarrow{k \to \infty} \rho^N_t$ weakly for each fixed $t$,
    where $\rho^{N,\delta}$ denotes the unique classical solution of \eqref{Liouville fo} with $F$ replaced by $F_{\delta}$ and $K$ replaced by $K_{\delta}$.
\end{definition}

\subsection{Preliminary Results} In this subsection we state and prove some results that will be useful in the paper. The first is a refinement of a compactness result borrowed from \cite[Proposition 1.2]{daudin2023optimal} for solutions $(m,\alpha)$ to the Fokker-Planck equation:
\be\label{FP}
\partial_t m+\text{div}(\alpha m)-\Delta m=0,
\ee
where $m\in \mathcal{C}([0,T],\mathcal{P}_2(\mathbb{T}^d))$ and $\alpha\in L^2_{dt\otimes m(t)}\left( [0,T]\times \mathbb{T}^d, \R^d\right)$.

\begin{proposition}\label{comp}
Assume that, for all $k\geq 1$, $(m_k,\alpha_k)$ solves the Fokker-Planck equation \eqref{FP} starting from $m_0$ and satisfies the uniform energy estimate
\be \label{est1}
\int_0^T\int_{\mathbb{T}^d} |\alpha_k(t,x)|^2dm_k(t)(x)dt\leq C,
\ee
for some constant $C>0$ independent of $k$. Then, for any $\delta\in (0,1)$, up to taking a subsequence, $(m_k,\alpha_km_k)$ converges in $\mathcal{C}^{\frac{1-\delta}{2}}\left( [0,T];\mathcal{P}_{2-\delta}(\mathbb{T}^d)\right)\times\mathcal{M}([0,T]\times\mathbb{T}^d,\R^d)$ to some $(m,w)$. The curve $m$ is in $\mathcal{C}^{1/2}\left( [0,T],\mathcal{P}_2(\mathbb{T}^d)\right)$, $w$ is absolutely continuous with respect to $m(t)\otimes dt$, for any $t_1,t_2\in [0,T]$ such that $t_1<t_2$ it holds that
\be \label{limitineq}
\int_{t_1}^{t_2}\int_{\mathbb{T}^d} \bigg| \frac{dw}{dm(t)\otimes dt}(t,x)\bigg|^2dm(t)(x)dt\leq \liminf_{k\rightarrow +\infty}\int_{t_1}^{t_2}\int_{\mathbb{T}^d}|\alpha_k(t,x)|^2dm_k(t)(x)dt
\ee
and $(m,\frac{dw}{dm(t)\otimes dt})$ solves \eqref{FP} starting from $m_0$.
\end{proposition}
 \begin{proof}
Let $w_n=\alpha_nm_n$. For the total variation $|w_n|$ we have
$$|w_n|\leq\int_0^T\int_{\mathbb{T}^d}\bigg| \frac{dw_n}{dm_n(t)\otimes dt}(t,x)\bigg|dm_n(t)(x)dt\leq \sqrt{T}\bigg(\int_0^T\int_{\mathbb{T}^d}\bigg| \frac{dw_n}{dm_n(t)\otimes dt}(t,x)\bigg|^2dm_n(t)(x)dt \bigg)^{1/2},$$
therefore $|w_n|\leq \sqrt{CT}$. We can therefore use Banach-Alaoglu theorem and, due to standard estimates for the Fokker-Planck equation, we can deduce that, for any $r\in (1,2)$,  $(m_n,w_n)$ converges in $\mathcal{C}([0,T],\mathcal{P}_r(\mathbb{T}^d))\times \mathcal{M}([0,T]\times \mathbb{T}^d,\R^d)$ to some element $(m,w)$. In fact, elementary arguments yield that the convergence also holds also in $\mathcal{C}([t_1,t_2],\mathcal{P}_r(\mathbb{T}^d))\times \mathcal{M}([t_1,t_2]\times \mathbb{T}^d,\R^d)$.

It is straightforward, because of the convergence, that $(m,w)$ satisfies the Fokker-Planck equation starting from $m_0$. By Theorem 2.34 of \cite{ambrosio2000functions} we discover that $w$ is absolutely continuous with respect to $m(t)\otimes dt$ and that \eqref{limitineq} holds. The bound that \eqref{limitineq} provides (when $t_1=0$ and $t_2=T$), yields $m\in \mathcal{C}^{1/2}\left( [0,T],\mathcal{P}_2(\mathbb{T}^d)\right)$, due to standard estimates for the Fokker-Planck equation.
 \end{proof}

 \begin{remark}
 As an artifact of the proof of the above proposition, we get that if $(m_k,\alpha_k)$ satisfies \eqref{est1} and the sequence $(m_k)_{k\in\mathbb{N}}$ is known to converge in $\mathcal{C}^{\frac{1-\delta}{2}}([0,T];\mathcal{P}_{2-\delta}(\mathbb{T}^d))$, then $w_k=\alpha_km_k$ converges in $\mathcal{M}([0,T]\times\mathbb{T}^d;\R^d)$ and \eqref{limitineq} holds as well.
 \end{remark}

We are also going to need the following two elementary lemmas, the first of which can be inferred \cite[Proposition 2.4.2]{dupuis2011weak} (notice that the boundness assumption on $f$ there is not necessary) and the second of which is standard.

 \begin{lemma}\label{lem.entvar}
     Let $X$ be a Polish space,  $\mu\in \mathcal{P}(X)$ and a function $f:X\rightarrow \R\cup \{-\infty,+\infty\}$ such that $e^f\in L^1(\mu)$. Then, the following equality holds
     \be \label{variational}
     \inf_{\nu\in \mathcal{P}(X)}\bigg\{ H(\nu|\mu)-\int_{X}f(x)d\nu(x) \bigg\}=-\log\bigg( \int_{X}e^{f(x)}d\mu(x)\bigg).
     \ee
     In particular, if $A \subset X$ is any Borel set, then 
     \begin{align} \label{generalvarformula}
         - \log \mu(A) = \inf_{\nu \in \cP(X), \, \nu(A) = 1} H(\nu | \mu).
     \end{align}
 \end{lemma}




\begin{lemma} \label{convergence lemma}
Let $(\mu_n)_{n\in\mathbb{N}}$ be a sequence of probability measures over $\mathbb{T}^{dN}$ with densities converging weakly in $L^1$ to the probability measure $\mu$. We also suppose that $(A_n)_{n\in\mathbb{N}}$ is an increasing sequence of subsets of $\mathbb{T}^{dN}$ converging to a set $A$. Then, $\lim_n\mu_n(A_n)=\mu(A)$.
\end{lemma}

\section{Proofs of the main results}

\subsection{Proof of Theorem \ref{thm.apriori}}
To prove Theorem \ref{thm.apriori}, we are going to need the following variational lower bound for $(0,K)$-entropy solutions.

\begin{proposition}\label{varformula}
Let the hypotheses of Theorem \ref{thm.apriori} hold. Let $\rho^N$ be an entropy solution of \eqref{Liouville fo} in the sense of Definition \ref{def entropy} and $A$ an open subset of $\mathbb{T}^{dN}$. Then, the following inequality holds
\be \label{varlowerbound}
-\frac{1}{N}\log \rho^{N}_T(A)\geq \inf_{\bm \alpha, f: f_T(A)=1}\bigg\{ \frac{1}{N}H(f_0|\rho_0^N)+\frac{1}{4N}\sum_{i=1}^N\int_0^T\int_{\mathbb{T}^{d}}|\alpha^{i}(t,\bx)|^2df_t( \bx)dt\bigg\},
\ee
where the infimum is taken over all $f \in C([0,T] ; \cP(\T^d))$ and $\bm \alpha=(\alpha^1,...,\alpha^N):[0,T]\times\mathbb{T}^{dN}\rightarrow \R^{dN}$ such that $f$ is a $(\bm \alpha,K)$-entropy solutions in the sene of Definition \ref{def entropy}.  
\end{proposition}

\begin{proof}
We give the proof only when $F = 0$ for simplicity, because the proof with $F \neq 0$ is the same but more notationally cumbersome.
\\
\\
\textit{Step 1 ($K$ is smooth and terminal condition is mollified).} Assume that $K$ is smooth. We consider the function $G$ with $G(x)=0$ if $x\in A$ and $G(x)=+\infty$ if $x\in A^c$, and, for $\delta\in (0,1)$, let $G_{\delta}$ be a smooth function such that $G_{\delta}(x) = 0$ on $A$ and $\frac{2}{\delta}\geq G_{\delta}(x) > 1/\delta$ if $\text{dist}(A) > \delta$.

Since $K$ is assumed to be smooth, \eqref{system} is strongly uniquely solvable, therefore we can consider its solution ${\bf X}_{\cdot}^N$ and its law (under $\mathbb{P}$) $\mathcal{L}({\bf X}_{\cdot}^N)$ in the path space $\mathcal{C}([0,T];\mathbb{T}^{dN})$. By using Lemma \ref{lem.entvar} and Girsanov's Theorem as in the proof of Theorem 4.1 of \cite{boue1998variational}\footnote{The only difference is that we have a random initial condition, which leads to the additional term in the right-hand side of \eqref{boue}, and the fact that we work on the torus, which is no problem because a version of \eqref{boue} on the torus easily follows from the corresponding version on $\R^d$.} , we have
\vspace{-1mm}
\begin{align} \label{boue}
    - \frac{1}{N} \log \mathbb{E}\bigg[e^{- G_{\delta}(X_T)} \bigg] &= \frac{1}{N}\inf_{Q\ll \mathcal{L}({\bf X}_{\cdot})}\bigg\{ H(Q|\mathcal{L}({\bf X}_{\cdot}))+\int _{\mathcal{C}([0,T];\mathbb{T}^{dN})}G_{\delta}(\omega_T)dQ(\omega)\bigg\} 
    \nonumber \\
    &= \frac{1}{N} \inf_{\bm \alpha, \bY} \bigg\{ H\big(\sL(\bY_0) | \sL(\bX_0) \big) + \mathbb{E}\Big[ \frac{1}{4} \sum_{i = 1}^N \int_0^T |\alpha_t^i|^2 dt + G_{\delta}(\bY_T) \Big] \bigg\}, 
\end{align}
where the second infimum is taken over all pairs consisting of a square-integrable adapted $\R^{dN}$-valued process $\bm \alpha = (\alpha^1,...,\alpha^N)$ and a continuous process $\bm Y$ satisfying
\begin{align*}
    dY_t^i = \bigg( \alpha_t^i + \frac{1}{N} \sum_{k \neq i} K(Y_t^i- Y_t^j) \bigg) dt + \sqrt{2} dW_t^i, \quad i = 1,...,N. 
\end{align*}

Furthermore, by the mimicking theorem \cite[Corollary 3.7]{brunick2013mimicking}, for any such $(\bm\alpha, \bY)$, there exists a measurable function $\bm\alpha=(\alpha^1,...,\alpha^N):[0,T]\times \T^{dN}\rightarrow \R^{dN}$ and a process $\tilde{\bm Y}$ such that $\tilde{\bm Y}$ is a weak solution of 
\begin{align*}
    d\tilde{Y}_t^i = \bigg( \alpha^i(t,\tilde{\bY}_t) + \frac{1}{N} \sum_{j \neq i} K(\tilde{Y}_t^i - \tilde{Y}_t^j) \bigg)dt + \sqrt{2} dW_t^i,\quad i=1,...,N,
\end{align*}
$\mathbb{E}\Big[\int_0^T |\alpha^i(t,\tilde{\bY}_t)|^2 dt \Big] \leq \mathbb{E}\Big[ \int_0^T |\alpha_t^i|^2 dt \Big]$, and for any $t\in [0,T]$, $\sL(\tilde{Y}_t) = \sL(Y_t)$. Moreover, it is clear that for any such $\tilde{\bY}$, its law $f_t = \sL(\tilde{\bm Y}_t)$ must be a weak solution of \eqref{Liouville fo}. 
Combining the last observations, we deduce that 

\begin{align} \label{alphafrep}
-\frac{1}{N}\log \mathbb{E}\bigg[ e^{-G_{\delta}({\bf X}_T)}\bigg] \geq  \frac{1}{N}\inf_{\bm \alpha,f}\bigg\{ H(f_0|\rho^N_0)+\frac{1}{4}\sum_{i=1}^N\int_0^T\int_{\mathbb{T}^{dN}}|\alpha^i(t,\bx)|^2df_t( \bx)dt+\int_{\mathbb{T}^{dN}}G_{\delta}(\bx)df_T( \bx)\bigg\}
\end{align} 
where the infimum is over pairs $(\bm \alpha, f)$ such that $f$ is a weak solution of \eqref{Liouville foc}. 

Our next goal is to argue that we can restrict the infimum in \eqref{alphafrep} to smooth solutions of \eqref{Liouville foc}. For this, we note that by considering the infimum in \eqref{alphafrep} first with respect to $f_0$ and then with respect to $\bm \alpha$, we see that we have 
\begin{align} \label{alphafrep2}
    -\frac{1}{N}\log \mathbb{E}\bigg[ e^{-G_{\delta}({\bf X}_T)}\bigg] =\frac{1}{N}\inf_{f_0 \ll \rho_0^N}\bigg\{ H(f_0|\rho^N_0) + \int_{\T^{dN}} V^{\delta}(\bx) df_0(\bx)\bigg\},
\end{align}
with $V^{\delta}(\bx)$ being the value function of a standard stochastic control problem, and in particular $V^{\delta}(\bx) = u^{\delta}(0,\bx)$, where $u^{\delta}$ solves 
\be\label{HJ1} 
-\partial u^{\delta} -\Delta u^{\delta} + |D_{\bx}u^{\delta} |^2 +\frac{1}{N}\sum_{i=1}^N\sum_{j\neq i}K(x^i-x^j)\cdot D_{x^i}u^{\delta} =0, \quad (t,\bx)\in [0,T)\times \T^{dN}
\ee
with terminal conditions $u^{\delta} (T,\bx)= G_{\delta}(\bx)$.
Moreover, the theory of stochastic control also tells us that the optimal feedback $\bm \alpha$ in \eqref{alphafrep} is independent of $f$, and takes the form $\alpha^{\delta,i}(t,\bx) = - 2 D_{x^i} u^{\delta}(t, \bx)$. By parabolic regularity and the smoothness of $K$ and $G_{\delta}$, $V^{\delta}$ and $\alpha^{\delta, i}$ are smooth. In particular, $f_0 \mapsto \int V^{\delta}(\bx) df_0(\bx)$ is continuous, so for any $\eps > 0$, the minimization problem \eqref{alphafrep2} admits a smooth $\eps-$minimizer $f_0^{\eps}$. If $f^{\eps}$ is the weak solution of \eqref{Liouville foc} driven by $K$ and $\bm  \alpha^{\delta}$ and starting from $f_0^{\eps}$, we deduce that $(\bm \alpha^{\delta}, f^{\eps})$ is a smooth $\eps$-minimizer for the infimum in \eqref{alphafrep}. In particular, this shows that \eqref{alphafrep} remains true when the infimum in the right-hand side is restricted to pairs $(\bm \alpha, f)$ such that $\bm \alpha$ is smooth and $f$ is a classical (hence entropy) solution of \eqref{Liouville foc}.
\newline \newline 
\textit{Step 2 ($K$ is smooth)}:
  The goal in this step is to take $\delta \to 0$ in \eqref{alphafrep}. By the previous step, we can find for each $\delta > 0$ a pair $(\bm \alpha^{\delta}, f^{\delta})$ such that $\bm \alpha^{\delta}$ is smooth, $f^{\delta}$ is a classical solution of \eqref{Liouville foc} and 
  \begin{align}
-\frac{1}{N}\log \mathbb{E}\bigg[ e^{-G_{\delta}({\bf X}_T)}\bigg]\geq\frac{1}{N}H(f_0^{\delta}|\rho^N_0)+\frac{1}{4N}\sum_{i=1}^N\int_0^T\int_{\mathbb{T}^{dN}}|\alpha^{i,\delta}(t,\bx)|^2df^{\delta}_t(\bx)dt+\frac{1}{N}\int_{\mathbb{T}^{dN}}G_{\delta}({ \bf x})df_T^{\delta}(\bx)-\delta.\label{3.1.4}
\end{align}
Using \eqref{Entineq} and applying Cauchy-Schwartz, it follows that
\begin{align}
H(f^{\delta}_t)+ &\int_0^t I(f^{\delta}_s) ds \label{3.1.5}  \leq
\sum_{i,j=1}^N\frac{1}{N}\bigg( \int_0^t\int_{\mathbb{T}^{dN}}\frac{|D_{x^i}f^{\delta}|^2}{f^{\delta}}d\bx ds\bigg)^{1/2}\bigg( \int_0^t\int_{\mathbb{T}^{dN}}|V(x^i-x^j)|^2df^{\delta}(\bx)ds\bigg)^{1/2} \nonumber \\
&+H(f_0^{\delta})+\sum_{i=1}^N\bigg( \int_0^t\int_{\mathbb{T}^{dN}}\frac{|D_{x^i}f^{\delta}|^2}{f^{\delta}}d\bx ds\bigg)^{1/2}\bigg(\int_0^t\int_{\mathbb{T}^{dN}}|\alpha^{i,\delta}(t,\bx)|^2df_t^{\delta}(\bx)dt\bigg)^{1/2}.
\end{align}
However, $G_{\delta}$ converges to $G$, so by the bounded convergence theorem, the left hand side of \eqref{3.1.4} converges and it is, therefore, bounded. Since $H(f_0^{\delta}|\rho_0^N)$, $G_{\delta}$ and $\frac{1}{N}\sum_{i=1}^N\int_0^t\int_{\mathbb{T}^{dN}}|\alpha^{i,\delta}(t,\bx)|^2df_t^{\delta}(\bx)dt$  are nonnegative, this implies that the terms on the right hand side of \eqref{3.1.4} are also uniformly bounded (independently of $\delta$). Since $H(f_0^{\delta}|\rho^N_0)$ is uniformly bounded, due to the upper bound of $\rho^N_0$, we also get that $H(f_0^{\delta})$ is uniformly bounded. Combining these facts with $V\in L^{\infty}$, we find that there exist constants $C_1,C_2$ with $C_2>0$, which are independent of $\delta$, such that \eqref{3.1.5} becomes
$$C_1+\int_0^t\int_{\mathbb{T}^{dN}}\frac{|D_{x^i}f^{\delta}|^2}{f^{\delta}}d\bx ds\leq C_2\bigg( \int_0^t\int_{\mathbb{T}^{dN}}\frac{|D_{x^i}f^{\delta}|^2}{f^{\delta}}d\bx ds\bigg)^{1/2}.$$
This clearly implies that $\int_0^t\int_{\mathbb{T}^{dN}}\frac{|D_{x^i}f^{\delta}|^2}{f^{\delta}}d\bx ds$, $i=1,...,N$, are also uniformly bounded. 

\noindent
We can, now, apply Proposition \ref{comp} in various ways. By virtue of the uniform boundness of $\int_0^T\int_{\mathbb{T}^{dN}}\frac{|D_{x^i}f^{\delta}|^2}{f^{\delta}}d\bx ds$, $\int_0^T\int_{\mathbb{T}^{dN}}|\alpha^{i,\delta}(t,\bx)|^2df_t^{\delta}(\bx)dt$, we get that as $\delta\rightarrow 0$, up to subsequences
\begin{align}
f^{\delta}\rightarrow f\;\; &\text{ in } \mathcal{C}([0,T];\mathcal{P}(\mathbb{T}^{dN})),\label{conv1}\\
\alpha^{i,\delta}f^{\delta}\rightarrow \alpha^if, \,\, D_{x^i}f^{\delta}\rightarrow D_{x^i}f\;\; &\text{ in }\mathcal{M}([0,t]\times \mathbb{T}^{dN};\R^{dN}),\; i=1,...,N, \;\; t \in (0,T]  \label{conv4}
\end{align}
In addition, by \eqref{3.1.5}, we derive that $H(f^{\delta}_t)$ is uniformly bounded independently of $t$, so $\int_0^TH(f^{\delta}(t))dt$ is uniformly bounded. The Vall\'ee–Poussin theorem \cite[Theorem 4.5.9]{bogachev2007measure}  implies that $f^{\delta}$ is uniformly integrable over $[0,T]\times \mathbb{T}^{dN}$, hence by the Dunford-Pettis theorem \cite[Theorem 4.7.18]{bogachev2007measure}, the convergence in \eqref{conv1} also holds weakly (again up to subsequences):
\begin{align}
f^{\delta}\rightarrow f &\text{ in } \mathcal{C}([0,T];\mathcal{P}(\mathbb{T}^{dN}))\text{ and weakly in }L^1([0,T]\times \mathbb{T}^{dN}),\\
&f_0^{\delta}\rightarrow f_0\text{ and }f_T^{\delta}\rightarrow f_T\text{ weakly in }L^1(\mathbb{T}^{dN}).
\end{align}
We now pass to the limit $\delta\rightarrow 0$ in \eqref{3.1.4}. By the weak convergence of $f_0^{\delta}$, the weak lower semicontinuity of the relative entropy, Proposition \ref{comp} and the nonnegativity of $G_{\delta}$ we deduce 
\be 
-\frac{1}{N}\log\rho^N_T(A)=-\frac{1}{N}\log \mathbb{E}\bigg[ e^{-G(X_T)}\bigg]\geq H(f_0|\rho^N_0)+\frac{1}{4N}\sum_{i=1}^N\int_0^T\int_{\mathbb{T}^{dN}}|\alpha^i(t,\bx)|^2df_t( \bx)dt,
\ee
hence in order to prove \eqref{varlowerbound}, it suffices to show that $(f,\alpha)$ is an admissible candidate for the infimum on its right hand side.
\vspace{2mm}

\noindent
We will start by showing that $f_T(A)=1$ or, equivalently, $f_T(A^c)=0$. Indeed, the family of sets $A_{\delta}=\{x\in\mathbb{T}^{dN}:G_{\delta}(x)\geq \frac{2}{\delta}\}$ is increasing and converges to $A^c$ as $\delta\rightarrow 0$. We have by Markov's inequality
$f_T^{\delta}(A_{\delta})\leq \frac{\delta}{2}\int_{\mathbb{T}^{dN}}G_{\delta}(\bx)df_T^{\delta}(\bx)$, thus $\lim_{\delta\rightarrow 0}f_T^{\delta}(A_{\delta})\leq 0,$
because $\int_{\mathbb{T}^{dN}}G_{\delta}(\bx)df_T^{\delta}(\bx)$ is uniformly bounded. But since $A_{\delta}$ is increasing and $f_T^{\delta}$ converges weakly in $L^1$ to $f_T$, Lemma \ref{convergence lemma} and the last inequality imply $f_T(A^c)=0$, which is what we wanted.
\vspace{2mm}

\noindent
We now prove that $f$ is an $(\bm \alpha,0,K)-$entropy solution. It is straightforward to check that the limit $(f,\bm \alpha)$ satisfies \eqref{Liouville foc} in the weak sense. Thus, it is an admissible candidate for the first infimum in \eqref{alphafrep}, hence
\vspace{-2mm}
\begin{align*}
H(f_0|\rho^N_0)+\frac{1}{4}\sum_{i=1}^N\int_0^T\int_{\mathbb{T}^{dN}}&|\alpha^i(t,\bx)|^2df_t( \bx)dt+\int_{\mathbb{T}^{dN}}G_{\delta}(\bx)df_T( \bx)+\delta \\
&\geq H(f_0^{\delta}|\rho^N_0)+\frac{1}{4}\sum_{i=1}^N\int_0^T\int_{\mathbb{T}^{dN}}|\alpha^{i,\delta}(t,\bx)|^2df^{\delta}_t(\bx)dt+\int_{\mathbb{T}^{dN}}G_{\delta}(\bx)df^{\delta}_T(\bx)
\end{align*}
We pass to the limit as $\delta\rightarrow 0$ and by the weak lower semi-continuity of the entropy, Proposition \ref{comp} and the fact that $f_T$ is supported on $A$, we get
\begin{align*}
    H(f_0|\rho^N_0)+\frac{1}{4}\sum_{i=1}^N\int_0^T&\int_{\mathbb{T}^{dN}}|\alpha^i(t,\bx)|^2df_t( \bx)dt\geq \liminf_{\delta\rightarrow 0}\bigg(H(f_0^{\delta}|\rho^N_0)+\frac{1}{4}\sum_{i=1}^N\int_0^T\int_{\mathbb{T}^{dN}}|\alpha^{i,\delta}(t,\bx)|^2df^{\delta}_t(\bx)dt\bigg)\\
    &\geq \liminf_{\delta\rightarrow 0}H(f_0^{\delta}|\rho^N_0)+ \liminf_{\delta\rightarrow 0}\frac{1}{4}\sum_{i=1}^N\int_0^T\int_{\mathbb{T}^{dN}}|\alpha^{i,\delta}(t,\bx)|^2df^{\delta}_t(\bx)dt\\
    &\geq H(f_0|\rho^N_0)+\frac{1}{4}\sum_{i=1}^N\int_0^T\int_{\mathbb{T}^{dN}}|\alpha^i(t,\bx)|^2df_t( \bx)dt.
\end{align*}
\vspace{-2mm}
We deduce that
\begin{align}
    \liminf_{\delta\rightarrow 0}H(f_0^{\delta}|\rho^N_0)&=H(f_0|\rho^N_0)\label{conv3}\\
    \liminf_{\delta\rightarrow 0}\frac{1}{4}\sum_{i=1}^N\int_0^T\int_{\mathbb{T}^{dN}}|\alpha^{i,\delta}(t,\bx)|^2df^{\delta}_t(\bx)dt&=\frac{1}{4}\sum_{i=1}^N\int_0^T\int_{\mathbb{T}^{dN}}|\alpha^i(t,\bx)|^2df_t( \bx)dt.\label{conv6}
\end{align}
Since $f^{\delta}$ is a smooth $(\bm \alpha^{\delta},K)$-entropy solution, \eqref{Entineq} holds for every $t\in [0,T]$ and can be rewritten as
\begin{align}
H(f_0^{\delta})+\frac{1}{2}\sum_{i=1}^N&\int_0^T\int_{\mathbb{T}^{dN}}|\alpha^{i,\delta}(s,\bx)|^2df_s^{\delta}(\bx)ds+\frac{1}{N}\sum_{i,j=1}^N\int_0^t\int_{\mathbb{T}^{dN}}V(x^i-x^j)\cdot D_{x^i}f^{\delta}(s,\bx)d\bx ds\nonumber\\
& \geq \frac{1}{2}\sum_{i=1}^N\int_t^T\int_{\mathbb{T}^{dN}}|\alpha^{i,\delta}(s,\bx)|^2df_s^{\delta}(\bx)ds+\frac{1}{2}\sum_{i=1}^N\int_0^t\int_{\mathbb{T}^{dN}}\bigg|\alpha^{i,\delta}(s,\bx)-\frac{D_{x^i}f^{\delta}}{f^{\delta}}\bigg|^2df^{\delta}_s(\bx)ds\nonumber\\
&\quad +H(f^{\delta}_t)+\frac{1}{2}\sum_{i=1}^N\int_0^t\int_{\mathbb{T}^{dN}}\frac{|D_{x^i}f^{\delta}|^2}{f^{\delta}}d\bx ds.\label{ineq}
\end{align}

\noindent
We observe that since $H(f_t^{\delta})$ is uniformly bounded, by passing to a further subsequence if necessary, we get $f_t^{\delta}\rightarrow f_t$ weakly in $L^1(\T^{dN})$. Due to the lower semi-continuity of the entropy, Proposition \ref{comp} and the remark after Proposition \ref{comp}, we get that the $\liminf_{\delta\rightarrow 0}$ of the right hand side of \eqref{ineq} is at least
\begin{align}
H(f_t)+&\frac{1}{2}\sum_{i=1}^N\int_0^t\int_{\mathbb{T}^{dN}}\bigg|\alpha^{i}(s,\bx)-\frac{D_{x^i}f}{f}\bigg|^2df_s( \bx)ds\nonumber\\
&+\frac{1}{2}\sum_{i=1}^N\int_t^T\int_{\mathbb{T}^{dN}}|\alpha^{i}(s,\bx)|^2df_s( \bx)ds+\frac{1}{2}\sum_{i=1}^N\int_0^t\int_{\mathbb{T}^{dN}}\frac{|D_{x^i}f|^2}{f}d\bx ds.\label{ineq2}
\end{align}
On the other hand, because of the convergences \eqref{conv4}, \eqref{conv3} and \eqref{conv6}, the left hand side of \eqref{ineq} converges to
\be\label{ineq3}
H(f_0)+\frac{1}{2}\sum_{i=1}^N\int_0^T\int_{\mathbb{T}^{dN}}|\alpha^{i}(s,\bx)|^2df_s( \bx)ds+\frac{1}{N}\sum_{i,j=1}^N\int_0^t\int_{\mathbb{T}^{dN}}V(x^i-x^j)\cdot D_{x^i}f(s,\bm x)d\bx ds,
\ee
so that $(f,\bm \alpha)$ satisfies \eqref{Entineq} for each $t\in [0,T]$. 
\\ 
\\
\noindent
\textit{Step 3 ($K \in \linf$).} Assume now that $K\in L^{\infty}$. We consider $K^r$ to be a family of mollifications of $K$. For any $r>0$, by step 2, \eqref{varlowerbound} holds when $K=K^r$. We denote by $V_r$ the right-hand side of \eqref{varlowerbound} with $K$ replaced by $K^r$; so that $-\frac{1}{N}\log \rho_T^{N,r}(A)\geq V_r$. We wish to show that $-\frac{1}{N}\log \rho_T^{N}(A)\geq V_0$. Of course if $\rho_T^N(A)=0$, this is trivial, so we may assume that $\rho_T^N(A)>0$.
\vspace{1mm}

\noindent
Note that since $\liminf_{r\rightarrow 0}\rho_T^{N,r}(A)\geq \rho_T^N(A)> 0$, the set $\{V_r |r\in (0,1)\}$ is bounded. Set 
$$\tilde{K}^{i,r}(\bx)=\frac{1}{N}\sum_{j\neq i}^NK^r(x^i-x^j)\text{ and } \tilde{K}^{i}(\bx)=\frac{1}{N}\sum_{j\neq i}^NK(x^i-x^j).$$
Then, for $(\bm \alpha^r,f^r)$ an $r$-minimizer for $V_r$ we have
\vspace{-2.5mm}
\begin{align*}
    V_r&\geq H_N(f_0^r|\rho_0^N)+\frac{1}{4N}\sum_{i=1}^N\int_0^T\int_{\T^{dN}}|\alpha^{i,r}(t,\bx)|^2df_t^r(\bx)dt-r\\
    &=H_N(f_0^r|\rho_0^N)+\frac{1}{4N}\sum_{i=1}^N\int_0^T\int_{\T^{dN}}|\alpha^{i,r}(t,\bx)+\tilde{K}^{i,r}(\bx)-\tilde{K}^i(\bx)|^2df_t^r(\bx)dt\\
    &\null\;-\frac{1}{4N}\sum_{i=1}^N\int_0^T\int_{\T^{dN}}|\tilde{K}^{i,r}(\bx)-\tilde{K}^i(\bx)|^2df_t^r(\bx)dt-\frac{1}{2N}\sum_{i=1}^N\int_0^T\int_{\T^{dN}}\alpha^{i,r}(t,\bx)\cdot (\tilde{K}^{i,r}(\bx)-\tilde{K}^i(\bx))df_t^r(\bx)dt.
\end{align*}
On the one hand, since $V_r$ is uniformly bounded, this implies 
\be\label{ext1}
\sup_{r\in(0,1)}H_N(f_0^r|\rho_0^N)+\frac{1}{4N}\sum_{i=1}^N\sup_{r\in (0,1)}\int_0^T\int_{\T^{dN}}|\alpha^{i,r}(t,\bx)|^2df_t^r(\bx)dt<\infty.
\ee
We note that due to the fact that $f^r$ is an $(\bm \alpha^r,K^r)$-entropy solution, $\| K^r\|_{L^{\infty}}\leq \|K\|_{L^{\infty}}$ and \eqref{ext1}, an argument as in Step 1 yields 
\be\label{ext3}
\sup_{t\in [0,T]}\sup_{r\in (0,1)}H(f^r_t)<\infty.
\ee
On the other hand, by Remark \ref{rem1}, $f^r$ is an $(\bm \alpha^r+\tilde{K}^r-\tilde{K},K)$-entropy solution, therefore 
\begin{align}
V_r\geq V_0
+\frac{1}{4N}\sum_{i=1}^N\bigg(\int_0^T\int_{\T^{dN}}|\tilde{K}^{i,r}(\bx)-\tilde{K}^i(\bx)|^2df_t^r(\bx)dt-2\int_0^T\int_{\T^{dN}}\alpha^{i,r}(t,\bx)\cdot (\tilde{K}^{i,r}(\bx)-\tilde{K}^i(\bx))df_t^r(\bx)dt\bigg).\label{ext2}
\end{align}
To finish the proof, we will show that the integral terms in \eqref{ext2} converge to $0$ as $r\rightarrow 0$. Indeed, for any $M>0$ and $i\in \{1,..,N\}$ we have
\begin{align*}
    &\int_0^T\int_{\T^{dN}}|\tilde{K}^{i,r}(\bx)-\tilde{K}^i(\bx)|^2df_t^r(\bx)dt=  \\ 
    & \hspace{3cm}\int_0^T\int_{\{f_t^r>M\}}|\tilde{K}^{i,r}(\bx)-\tilde{K}^i(\bx)|^2df_t^r(\bx)dt +\int_0^T\int_{\{f_t^r\leq M\}}|\tilde{K}^{i,r}(\bx)-\tilde{K}^i(\bx)|^2df_t^r(\bx)dt 
    \\
    &\leq  \bigg( \int_0^T\int_{\{f_t^r>M\}}\frac{|\tilde{K}^{i,r}(\bx)-\tilde{K}^i(\bx)|^4}{\log f_t^r }df_t^r(\bx)dt   \bigg)^{1/2} \left( \int_0^TH(f^r_t)dt\right)^{1/2} +M\int_0^T\int_{\T^{dN}}|\tilde{K}^{i,r}(\bx)-\tilde{K}^i(\bx)|^2d\bx dt.
\end{align*}
Therefore, there exists a constant $C>0$ depending on the bound provided by \eqref{ext3} and $\| K\|_{L^{\infty}}$ such that
\begin{align*}
    \int_0^T\int_{\T^{dN}}|\tilde{K}^{i,r}(\bx)-&\tilde{K}^i(\bx)|^2df_t^r(\bx)dt\leq \frac{C}{\sqrt{\log M}}+MT\int_{\T^{dN}}|\tilde{K}^{i,r}(\bx)-\tilde{K}^i(\bx)|^2d\bx
\end{align*}
Now $K\in L^{\infty}$, so $\tilde{K}^{i,r}\xrightarrow{r\rightarrow 0} \tilde{K}^i$ in $L^2$, hence by passing to the limit we discover
$$0\leq \limsup_{r\rightarrow 0}\int_0^T\int_{\T^{dN}}|\tilde{K}^{i,r}(\bx)-\tilde{K}^i(\bx)|^2df_t^r(\bx)dt\leq \frac{C}{\sqrt{\log M}}$$
and since $M$ was arbitrary, we get
\be\label{ext4}
\lim_{r\rightarrow 0}\int_0^T\int_{\T^{dN}}|\tilde{K}^{i,r}(\bx)-\tilde{K}^i(\bx)|^2df_t^r(\bx)dt=0.
\ee
We also have by Cauchy-Schwartz
\begin{align*}
    \bigg|\int_0^T\int_{\T^{dN}}\alpha^{i,r}&(t,\bx)\cdot (\tilde{K}^{i,r}(\bx)-\tilde{K}^i(\bx))df_t^r(\bx)dt\bigg|\\
    &\leq\bigg( \int_0^T\int_{\T^{dN}}|\alpha^{i,r}(t,\bx)|^2df_t^r(\bx)dt  \bigg)^{1/2}\bigg( \int_0^T\int_{\T^{dN}}|\tilde{K}^{i,r}(\bx)-\tilde{K}^i(\bx)|^2df_t^r(\bx)dt  \bigg)^{1/2}.\\
\end{align*}
By \eqref{ext1} and \eqref{ext4} we get
\be\label{ext5}
\lim_{r\rightarrow 0}\int_0^T\int_{\T^{dN}}\alpha^{i,r}(t,\bx)\cdot (\tilde{K}^{i,r}(\bx)-\tilde{K}^i(\bx))df_t^r(\bx)dt=0.
\ee
Now we send $r$ to $0$ in \eqref{ext2} and use \eqref{ext4}, \eqref{ext5} to to conclude that 
\begin{align*}
    -\frac{1}{N}\log \rho^N_T(A)\geq -\frac{1}{N}\limsup_{r\rightarrow 0}\log \rho_T^{N,r}(A)\geq \limsup_{r\rightarrow 0}V_r\geq V_0.
\end{align*}
\end{proof}




\begin{proof}[Proof of Theorem \ref{thm.apriori}]
The first step will be to reinterpret \cite[Theorem 2]{FournierGuillin} in a convenient way. If we specialize this result to the torus, we find that for each $p$ there is a constant $C \geq 1$ depending only on $p$ and $d$ such that whenever $X^1,X^2,...$ are i.i.d. $\T^d$-valued random variables with common law $m$, we have $\bP\Big[\bd_p(m_{\bX}^N, m) > \epsilon \Big] \leq C \exp( - \frac{a_p(\eps) N}{C}).$
In other words, 
\begin{align*}
    - \frac{1}{N} \log \bP\Big[ {\bf d}_p(m_{\bX}^N, m) > \epsilon \Big] \geq \frac{a_p(x)}{C} - \log(C)/N.
\end{align*}
Recalling the variational formula \eqref{generalvarformula} from Lemma \ref{lem.entvar}, we find that the implication
\begin{align} \label{fgrewritten}
    \frac{1}{N} H(Q | m^{\otimes N}) < \frac{a_p(\eps)}{C} - \log(C)/N \implies Q(A^{m,p}_{N,\epsilon}) < 1
\end{align}
holds for $Q \in \cP(\T^{dN})$, where $A^{m,p}_{N,\eps} = \Big\{\bx \in (\T^d)^N \Big| {\bf d}_p(m_{\bx}^N, m) > \eps\Big\}$.

We are now going to combine \eqref{fgrewritten} with Proposition \ref{varformula} to complete the proof. Suppose that we have a pair $f^N$ and $\bm \alpha$ such that $f^N$ is an entropy solution of \eqref{Liouville foc}. Suppose further that
\begin{align*}
    H_N(f^N_0 | \ov{\rho}^N_0) + \frac{1}{4N} \sum_{i = 1}^N \int_0^T \int_{\T^d} |\alpha^i(t,\bx)|^2 df^N_t(\bx) dt < \frac{a_p(\eps)}{\cconc} - \frac{\log{\cconc}}{N}.
\end{align*}
holds for some $\cconc > 0$. By assumption, it follows that 
\begin{align*}
     H_N(f^N_T | \ov{\rho}^N_T) \leq \cent \bigg(\frac{1}{N} +a_p(\eps)/\cconc - \log \cconc/N \bigg) = \frac{\cent}{\cconc}a_p(\eps) - \frac{\cent\log \cconc - \cent}{N}.
\end{align*}
Some simple algebra shows that if $\cconc > \max\{C \cent, eC\}$, then (here we use $\cent \geq 1$) 
\begin{align*}
    \frac{\cent}{\cconc} a_p(\eps) - \frac{\cent\log \cconc - \cent}{N} \leq Ca_p(\eps) - \log (C)/N.
\end{align*}
In particular, setting $\cconc = eC\cent$ (i.e. $\cconc = C_d \cent$ with $\cconc = eC)$, and applying \eqref{fgrewritten}, we find that we have the implication
\begin{align*}
    H_N(f^N_0 &| \ov{\rho}^N_0) + \frac{1}{4N} \sum_{i = 1}^N \int_0^T \int_{\T^d} |\alpha^i(t,\bx)|^2 df^N_t(\bx) dt  < \frac{a_p(\eps)}{\cconc} - \log{\cconc}/N \\
    &\implies  H_N(f^N_T | \ov{\rho}^N_T) \leq \frac{a_p(\eps)}{C} - \log (C) /N \implies f^N_T(A_{N,\epsilon}) < 1.
\end{align*}
In light of Proposition \ref{varformula}, this implies $- \frac{1}{N} \log \rho_T^N(A_{N,\epsilon}) \geq a_p(\eps)/\cconc - \log(\cconc)/N$,
or in other words $\rho_T^N(A_{N,\epsilon}) \leq \cconc \exp (- \frac{\epsilon^d N}{\cconc}).$
\end{proof}

\subsection{Proof of Proposition \ref{prop.wminus}}

We begin by establishing the well-posedness of \eqref{limit pde}.

\begin{proposition} \label{prop.limitpde}
    Suppose Assumption \ref{assump.main} holds. Then \eqref{limit pde} admits a unique classical solution, which satisfies 
    \begin{align*}
        \|\ov{\rho}\|_{C^{2,\beta}_{t,x}} \leq C, \quad \inf_{t,x} \ov{\rho}(t,x) \geq C^{-1}, 
    \end{align*}
    with $C$ depending only $T$, $d$, and the constants $C_0$ and $\beta$ appearing in Assumption \ref{assump.main}.
\end{proposition}

\begin{proof}
    The main challenge is to obtain appropriate a-priori estimates, so we explain this point in detail and then quickly sketch the existence and uniqueness part. Thus we assume for the moment that in addition to Assumption \ref{assump.main}, $K$ is smooth, so that we have a unique classical solution $\ov{\rho}$. In what follows, $C$ can increase from line to line but can depend freely on the constants indicated in the statement of the proposition, and dependence on other parameters will be clearly indicated, e.g. $C(p)$ indicates a constant which can depend on $p$ as well as the constants appearing in the statement of the proposition. Moreover, we let $\phi$ be a vector field with $\text{div} \phi = \text{div} K$ and $\|\phi\|_{\infty} = \|\text{div} K\|_{-1,\infty}$. First, we have by integration parts and Young's inequality
    \begin{align*}
        \frac{d}{dt} H(\ov{\rho}_t)
        &= - I(\ov{\rho}_t) + \int_{\T^d} D \ov{\rho} \cdot \big( F + \phi * \ov{\rho} \big) dx
        \\
        &\leq - \frac{1}{2}  I(\ov{\rho}_t) + \frac{1}{2} \int_{\T^d} |F + \phi * \ov{\rho} |^2 \rho dx \leq C (\|F\|^2_{\infty} + \|\text{div} K\|^2_{-1,\infty}), 
    \end{align*}
    and so in particular $\int_0^T I(\ov\rho_t) dt \leq C$, from which, by Cauchy-Schwartz, it follows that $\|D \rho\|_{L^1_{t,x}} \leq C$. Rewriting \eqref{limit pde} in non-divergence form as 
    \begin{align} \label{nondiv}
        \partial_t \ov{\rho} = \Delta \ov{\rho} - D \ov{\rho} \cdot \big(K * \rho \big) - \ov{\rho}  \phi *  D \ov{\rho} , 
    \end{align}
    a standard argument using the maximum principle shows that for each $(t,x)$, we have 
    \begin{align*}
        \ov{\rho}(t,x)  \leq \|\rho_0\|_{\infty} \exp \Big( \int_0^t \|\phi *  D \ov{\rho}\|_{L_x^{\infty}} ds \Big) \leq \|\rho_0\|_{\infty} \exp \Big( \int_0^t \|\phi\|_{\linf} \|D \ov{\rho}_s\|_{L_x^{1}} ds \Big) \leq C,  
    \end{align*}
    and likewise 
    \begin{align*}
    \ov{\rho}(t,x)  \geq  \inf_x \rho_0(x) \exp \Big( -\int_0^t \|\phi *  D \ov{\rho}\|_{L_x^{\infty}} ds \Big) \geq \inf_x \rho_0(x) \exp \Big( - \int_0^t \|\phi\|_{\linf} \|D \ov{\rho}_s\|_{L_x^{1}} ds \Big) \geq C^{-1}. 
    \end{align*}
    Thus we have $\|\ov{\rho}\|_{\linf} \leq C$, and $\inf_{t,x} \ov{\rho} \geq C^{-1}$. From here, we view \eqref{nondiv} as a perturbation of the heat equation, applying the standard Calderon-Zygmund estimates and then the Gagliardo-Nirenberg interpolation inequality to get for any $p < \infty$, 
    \begin{align*}
        \norm{\ov{\rho}}_{W_{t,x}^{2,p}} &\leq C(p) \Big( \|\ov \rho_0\|_{W_x^{2,p}} + \| \phi * D \ov{\rho}\|_{L^p_{t,x}} \Big) \leq C(p) \Big( 1 + \|D \ov{\rho}\|_{L^p_{t,x}} \Big) \\ & \leq C(p)\Big(1 + \|\ov{\rho}\|_{L_{t,x}^p}^{1/2}\|\ov{\rho}\|_{W_{t,x}^{2,p}}^{1/2} \Big) 
        \leq C(p) + \frac{1}{2} \norm{\ov{\rho}}_{W_{t,x}^{2,p}}.
    \end{align*}
    Thus for any $p < \infty$, $\norm{\ov{\rho}}_{W_{t,x}^{2,p}} \leq C(p)$. Choosing a large enough $p$, we get by Sobolev embeddings $\norm{\ov{\rho}}_{C_{t,x}^{2,\beta}} \leq C$, and then by again viewing \eqref{nondiv} as a perturbation of the heat equation, we get by the Schauder estimates 
    \begin{align*}
        \norm{\ov{\rho}}_{C_{t,x}^{2,\beta}} \leq C \Big(\norm{\rho_0}_{C^{2,\beta}} + \norm{ D \ov{\rho}}_{C_{t,x}^{\beta}} \norm{K * \rho}_{C_{t,x}^{\beta}} + \|\rho\|_{C^{\beta}_{t,x}} \|\phi * D \ov{\rho}\|_{C_{t,x}^{\beta}} \Big) \leq C \Big(1 + \|\ov{\rho}\|_{C_{t,x}^{1,\beta}}^2 \Big) \leq C. 
    \end{align*}
    Thus we have established that when $K$ is smooth, the unique classical solution of \eqref{limit pde} satisfies the estimates stated in the proposition. From this a-priori estimate, a standard mollification and compactness argument can be used to obtain the existence of a solution satisfying the desired bounds when $K$ is not smooth but Assumption \ref{assump.main} is in force. Uniqueness of classical solutions, meanwhile, can be easily proved in a number of ways, e.g. given two smooth solutions $\ov{\rho}_t^1$ and $\ov{\rho}_t^2$ one can compute $\frac{d}{dt} H(\ov{\rho}_t^1 | \ov{\rho}_t^2)$ and conclude via Grownall's inequality. We omit the details.
\end{proof}

The proof of Proposition \ref{prop.wminus} also requires the following lemma, which is an easy extension of the main quantitative estimate of \cite{jabin2018quantitative} to the ``controlled" Liouivlle equation \eqref{Liouville foc}.

\begin{lemma} \label{lem.controlledentropyest}
Let Assumption \ref{assump.main} hold, let $\ov{\rho}$ be the unique classical solution of \eqref{limit pde} provided by Proposition \ref{prop.limitpde}, and let $f$ be an entropy solution of \eqref{Liouville foc} in the sense of Definition \ref{def entropy}. There exists a constant $\cent$ depending only on $T$, $d$, and the constants $\beta$ and $C_0$ in Assumption \ref{assump.main} such that
\be 
H_N(f_T|\overline{\rho}^N_T)\leq H_N(f_0|\overline{\rho}^N_0)+ \cent\bigg(\frac{1}{N}+  \frac{1}{4N}\sum_{i=1}^N\int_0^T\int_{\mathbb{T}^d}|\alpha^i(t,\bx)|^2df_t(\bx)dt\bigg).
\ee
\end{lemma}

\begin{proof}
    The proof follows closely the proof of \cite[Theorem 1]{jabin2018quantitative}, and so we report only the main difference. The first step is to mimic \cite[Lemma 2]{jabin2018quantitative} (here it is crucial that we work with an entropy solution $f$) to get
    \begin{align*}
        H_N(f_t | \ov{\rho}^N_t) &\leq H_N(f_0 | \ov{\rho}^N_0) - \frac{1}{N^2} \sum_{i,j = 1}^N \int_0^t \int_{\T^{dN}}  \big(K(x^i - x^j) - K * \ov{\rho}(x^i) \big) \cdot D_{x^i} \log \ov{\rho}^N df_s(\bx) ds \\
        &\qquad - \frac{1}{N^2} \sum_{i,j = 1}^N \int_0^t \int_{\T^{dN}}  \bigg(\text{div}_{x^i} K(x^i - x^j) - \text{div}_{x^i} K * \ov{\rho}(x^i) \bigg) df_s(\bx) ds 
        \\
        &\qquad +  \frac{1}{N} \sum_{i = 1}^N \int_0^t \int_{\T^{dN}}  \alpha^i \cdot D_{x^i} \log \frac{f}{\ov{\rho}^N} df_s(\bx) ds - \frac{1}{N} \sum_{i = 1}^N \int_0^t \int_{\T^{dN}} \Big| D_{x^i} \log \frac{f}{\ov{\rho}^N} \Big|^2 df_s(\bx) ds.
    \end{align*}
    Young's inequality immediately gives 
    \begin{align} \label{gronwall}
        H_N(f_t | \ov{\rho}^N_t) &\leq H_N(f_0 | \ov{\rho}^N_0) - \frac{1}{N^2} \sum_{i,j = 1}^N \int_0^t \int_{\T^{dN}}  \bigg(K(x^i - x^j) - K * \ov{\rho}(x^i) \bigg) \cdot D_{x^i} \log \ov{\rho}^N df_s(\bx) ds \nonumber \\
        &- \frac{1}{N^2} \sum_{i,j = 1}^N \int_0^t \int_{\T^{dN}}  \bigg(\text{div}_{x^i} K(x^i - x^j) - \text{div}_{x^i} K * \ov{\rho}(x^i) \bigg) df_s(\bx) ds 
        \nonumber \\
        &  - \frac{1}{2N} \sum_{i = 1}^N \int_0^t \int_{\T^{dN}}  \Big| D_{x^i} \log \frac{f}{\ov{\rho}^N} \Big|^2 df_s(\bx) ds + \frac{1}{2N} \sum_{i = 1}^N \int_0^t \int_{\T^{dN}}  \big|\alpha^i(t,\bx)\big|^2 df_s(\bx) ds.
    \end{align}
    Notice that up to the last term and the factor $1/2$ appearing in the penultimate term, this is the same inequality appearing in \cite[Lemma 2]{jabin2018quantitative}. We now follow exactly the proof of \cite[Theorem 1]{jabin2018quantitative}, applying Lemmas 3 and 4 of \cite{jabin2018quantitative} (which are easily seen to apply here despite the fact that $f$ satisfies the ``perturbed" Liouville equation \eqref{Liouville foc} rather than the original Liouville equation \eqref{Liouville fo}) to bound the second and third terms appearing on the right-hand side of \eqref{gronwall}. This results in the bound 
    \begin{align*}
        H_N(f_t | \ov{\rho}^N_t) \leq H_N(f_0 | \ov{\rho}^N_0) + C \int_0^t \bigg(H_N(f_s | \ov{\rho}^N_s) + \frac{1}{N} \bigg) ds + \frac{1}{N} \sum_{i = 1}^N \int_0^t \int_{\T^{dN}}  \big|\alpha^i(s,\bx)\big|^2 df_s(\bx) ds, 
    \end{align*}
    with $C$ depending only on the constants indicated in the lemma. An application of Gronwall's inequality completes the proof.
\end{proof}

\begin{proof}[Proof of Proposition \ref{prop.wminus}]
    The existence of an admissible entropy solution is proved already in \cite[Proposition 1]{jabin2018quantitative}. Let $\rho^{N, \delta}$ be the unique classical solutions of \eqref{Liouville fo} with $K_{\delta}$ replacing $K$, and $\ov{\rho}^{\delta}$ be the unique classical solution of \eqref{limit pde} with $K_{\delta}$ replacing $K$. By Theorem \ref{thm.apriori} and Lemma \ref{lem.controlledentropyest}, we have 
    \begin{align} \label{concentrationliouvilledelta}
    \rho^{N,\delta}_t(A^{p,\delta}_{N,\e}) \leq \cconc \exp(- \cconc^{-1} a_p(\eps) N), \text{ where }A^{p,\delta}_{N,\e}=\bigg\{ x\in \mathbb{T}^{dN}\bigg| {\bf d}_p(m^N_{\bx},(\overline{\rho}^{\delta})^{\otimes N}_t)>\e \bigg\},
\end{align}
with $\cconc$ depending only on the constants stated in the proposition (here we use the fact that $\|K_{\delta} \|_{W^{-1,\infty}} \leq \|K\|_{W^{-1,\infty}}$ and $\|\text{div} \, K_{\delta} \|_{W^{-1,\infty}} \leq \|\text{div} \, K \|_{W^{-1,\infty}}$). Let $\delta_k$ be the sequence appearing in the definition of admissible entropy solution. Notice that by the uniqueness part of Proposition \ref{prop.limitpde}, we must have $\ov{\rho}^{\delta_k}_t \to \ov{\rho}_t$ for each fixed $t$. Notice also that for $k$ large enough, we will have $A_{N,\eps}^{p,\delta_k} \supset A^p_{N,2\eps}$, so that 
\begin{align*}
    \rho^N_T(A^p_{N, 2\eps}) \leq \liminf_{k \to \infty} \rho^{N,\delta_k}_T(A^p_{N,2\eps}) \leq \liminf_{k \to \infty} \rho^{N,\delta_k}_T(A_{N, \eps}^{p,\delta_k}) \leq \cconc \exp(-\cconc^{-1} a_p(\eps) N), 
\end{align*}
which, after replacing $\cconc$ by $2\cconc$, completes the proof.
\end{proof}

\bibliographystyle{alpha}		
\bibliography{singular}

\end{document}